\documentclass[12pt]{amsart}
\usepackage{mathrsfs}
\usepackage{txfonts}
\usepackage{amssymb}
\usepackage{color}
\makeatletter \@mparswitchfalse \makeatother
\newtheorem{question}{Question}[section]
\newtheorem{theorem}{Theorem}[section]
\newtheorem{lemma}{Lemma}[section]
\newtheorem{corollary}{Corollary}[section]

\newtheorem{definition}{Definition}[section]

\def\eqref#1{(\ref{eq#1})}

\numberwithin{equation}{section}
\begin{document}
\title{Partial factorization  and reflexivity   of operator algebras}
 \author[F. Jia]{Fengyang Jia}\author[G. Ji]{Guoxing Ji$^*$}\thanks{$^*$Corresponding author}
  \address{School  of Mathematics and Statistics,
  Shaanxi Normal University,
  Xian , 710119, People's  Republic of  China}
 \email{jfy123@snnu.edu.cn} \email{gxji@snnu.edu.cn}

     \thanks{This research was
supported by the National Natural
   Science Foundation of China(No. 12271323).
}
    \subjclass{Primary    47L35, 47L30; Secondary 46K50}
\keywords{operator algebra, transitive algebra, partial factorization, nest, Reflexivity} \maketitle
\begin{abstract}
 Let $\mathcal{H}$ be a separable infinite dimensional Hilbert space and $\mathcal{B}(\mathcal{H})$ the algebra of all bounded linear operators on $\mathcal{H}$. A subalgebra $\mathfrak{A}$ in $\mathcal{B}(\mathcal{H})$ has the left (resp.\ right) partial factorization property if for any invertible operator $S\in\mathcal{B}(\mathcal{H})$, there exists an isometry (resp.\ a co-isometry) $U\in\mathcal{B}(\mathcal{H})$ such that $U^*S, S^{-1}U\in\mathfrak{A}$. We show that if $\mathfrak{A}$ is weak operator topology closed with the left (resp.\ right) partial factorization property, then $\mathfrak{A}$ is the nest algebra associated with its invariant subspace lattice. In particular, if $\mathfrak{A}$ is transitive, then $\mathfrak{A}=\mathcal{B}(\mathcal{H})$. This gives a positive answer to Question 6.3 raised by B.V.R. Bhat and M. Kumar in \emph{Publ. Res. Inst. Math. Sci.} \textbf{60}(2024), 507--537.
\end{abstract}

\section{Introduction}Let \( \mathcal{H} \) be a separable infinite-dimensional Hilbert space, and let \( \mathcal{B}(\mathcal{H}) \) denote the algebra of all bounded linear operators on \( \mathcal{H} \). One of the most fundamental open problems in operator theory is the invariant subspace problem: does every operator in \( \mathcal{B}(\mathcal{H}) \) admit a nontrivial invariant subspace?
A closely related famous open question is the transitive algebra problem: if \( \mathfrak{A} \) is a closed  unital  subalgebra of \( \mathcal{B}(\mathcal{H}) \)  in  the weak operator topology with only trivial invariant subspaces, does it necessarily follow that \( \mathfrak{A} = \mathcal{B}(\mathcal{H}) \)? These classical problems have been extensively studied over decades and remain unsolved to this day (cf. \cite{chengg,hou1,hou2,rad}). Very recently, new progress has been made on Halmos' third problem; see \cite{chengl,jiang}.
On the other hand, it is of great interest to determine the invariant subspace lattice of a unital subalgebra \( \mathfrak{A} \) of \( \mathcal{B}(\mathcal{H}) \). Suppose that \( \mathfrak{A} \subseteq \mathcal{M} \) for some von Neumann algebra \( \mathcal{M} \). A particularly interesting topic is to investigate the relative invariant subspace lattice of \( \mathfrak{A} \), which consists of all projections in \( \mathcal{M} \) whose ranges are invariant under \( \mathfrak{A} \).
For instance, Gilfeather and Larson \cite{gil} characterized the relative invariant subspace lattice of the nest subalgebra \( \operatorname{Alg}\mathcal{N}\cap \mathcal{M} \) associated with a nest \( \mathcal{N} \) in \( \mathcal{M} \). Bhat and Kumar \cite{bha} proved that if \( \mathfrak{A}\subseteq \mathcal{M} \) is logmodular (respectively, has the factorization property) in the sense that the set \( \{A^*A: A, A^{-1} \in\mathfrak{A}\} \) is dense in (respectively, coincides with) the set of all positive invertible elements of \( \mathcal{M} \), then the relative invariant subspace lattice of \( \mathfrak{A} \) is commutative. Furthermore, if \( \mathcal{M} \) is a factor, then this lattice is a nest. In particular, they proposed the following question.
\begin{question}(\cite[Question 6.3]{bha}) \label{q1}
Is every weakly closed algebra with the factorization property in \( \mathcal{B}(\mathcal{H}) \) automatically reflexive? In particular, is every weakly closed transitive algebra with the factorization property equal to \( \mathcal{B}(\mathcal{H}) \)?
\end{question}
 We completely answer this open question in the present paper. Specifically, we prove that if \( \mathfrak{A}\subseteq\mathcal B(\mathcal H) \) is closed in the weak operator topology and has  the left (resp.\ right) partial factorization property, then \( \mathfrak{A} \) coincides with the nest algebra associated with its invariant subspace lattice. In particular, if \( \mathfrak{A} \) is transitive, then \( \mathfrak{A} = \mathcal{B}(\mathcal{H}) \). To proceed, we first recall some necessary definitions and preliminary notions.

  For any \( x,y\in \mathcal{H} \), let \( x\otimes y \) denote the rank-one operator on \( \mathcal{H} \) defined by \( x\otimes y(z)=\langle z,y\rangle x \) for all \( z\in\mathcal{H} \). We write \( \mathcal{F}(\mathcal{H}) \) and \( \mathcal{K}(\mathcal{H}) \) for the sets of all finite-rank operators and compact operators on \( \mathcal{H} \), respectively. Let $\pi$ denote the  canonical quotient map from $\mathcal B(\mathcal H)$ onto  the Calkin algebra \( \mathcal C(\mathcal H)=\mathcal{B}(\mathcal{H})/\mathcal{K}(\mathcal{H}) \). For any $A\in \mathcal B(\mathcal H)$, let \( \sigma(A) \) and \( \sigma_e(A) \) stand for the spectrum and essential spectrum of \( A \), respectively. Throughout this paper, \( I \) denotes the identity operator on any Hilbert space without ambiguity. We identify each closed subspace \( M \subseteq \mathcal{H} \) with the corresponding orthogonal projection \( P_M \in \mathcal{B}(\mathcal{H}) \) from \( \mathcal{H} \) onto \( M \).  
Given a subset \( \mathcal{S}\subseteq\mathcal{B}(\mathcal{H}) \), the commutant of \( \mathcal{S} \) is defined as
\[
\mathcal{S}^{\prime}=\{B\in\mathcal{B}(\mathcal{H}): AB=BA \text{ for all } A\in\mathcal{S}\}.
\]
For a von Neumann algebra \( \mathcal{M} \), its center is given by \( \mathcal{Z}(\mathcal{M})=\mathcal{M}\cap\mathcal{M}^{\prime} \). We denote by \( P(\mathcal{M}) \) and \( \mathcal{M}^{-1} \) the set of all projections and the set of all invertible elements in \( \mathcal{M} \), respectively.
Let \( \mathfrak{A} \) be a unital subalgebra of \( \mathcal{M} \). The relative invariant subspace lattice of \( \mathfrak{A} \) in \( \mathcal{M} \) is defined as
\[
\operatorname{Lat}_{\mathcal{M}}\mathfrak{A}=\{E\in\mathcal P(\mathcal{M}): E^{\perp}AE=0 \text{ for all } A\in\mathfrak{A}\}.
\]
 When $\mathcal M=\mathcal B(\mathcal H)$, we  write $\operatorname{Lat}\mathfrak A$  for  the invariant subspace lattice of $\mathfrak A$. The following notion was introduced by Pitts in \cite{pit}.
\begin{definition}(\cite{pit})
Let \( \mathfrak{A} \) be a unital subalgebra of \( \mathcal{M} \), and let \( S\in\mathcal{M}^{-1} \). If there exists an isometry (resp.\ a co-isometry) \( U\in\mathcal{M} \) such that \( U^*S, S^{-1}U\in\mathfrak{A} \), then \( S \) is said to admit a left (resp.\ right) partial factorization in \( \mathfrak{A} \). If every invertible element in \( \mathcal{M}^{-1} \) admits a left (resp.\ right) partial factorization in \( \mathfrak{A} \), then \( \mathfrak{A} \) is said to have the left (resp.\ right) partial factorization property.
\end{definition}
Furthermore, if the above isometry/co-isometry \( U \) is unitary, then \( S \) is said to admit a factorization in \( \mathfrak{A} \). Accordingly, \( \mathfrak{A} \) is said to have the factorization property if every invertible element in \( \mathcal{M}^{-1} \) admits a factorization in \( \mathfrak{A} \). For a \( \sigma \)-weakly closed subalgebra \( \mathfrak{A} \), it is known from \cite[Proposition 4.4]{pit} that an invertible  operator \( S \) admits a factorization in \( \mathfrak{A} \) if and only if \( S \) admits both left and right partial factorizations in \( \mathfrak{A} \). We show that if $\mathfrak A$ has the left(resp.\ right) partial factorization property, then $\mathfrak A$ is the nest algebra associated  to its invariant subspace lattice.
  \section{ Transitive algebras  with   the  partial factorization property}
   We  give a positive  answer for the second part of the  Question \ref{q1} in this section.  To consider the  answer for the first part of this question, we give a slight generalization. We  first give  several auxiliary lemmas.

    Let $\{\mathcal H_n:n\geq1\}$ be a finite or countable sequence of separable complex Hilbert spaces.   Put
   \begin{equation}\label{f2001}
   \mathcal H=\oplus_{n\geq1} \mathcal H_n  \text { and }\mathfrak D=\oplus_{n\geq1} \mathcal B(\mathcal H_n). \end{equation}  For every $n$,
   denote  by $E_n$ the projection from $\mathcal H$ onto $\mathcal H_n$ and  let
$\mathcal E_n:\mathfrak D\to \mathcal B(\mathcal H_n)$ be the coordinate map $\mathcal E_n(A)=A_n$ for any $A=\oplus_{n\geq1}A_n \in\mathfrak D$.
Whenever an operator on $\mathcal H_n$ is regarded as supported on the $n$th
coordinate, it is extended by zero to all other summands. Define the
norm-closed two-sided ideal
\begin{equation}\label{f2002}
 \mathcal J_n=\{A\in\mathfrak D:A_j=0\ (j\ne n),\ A_n\in\mathcal K(\mathcal H_n)\}.
\end{equation}
Fix an $n$ and write
\begin{equation}\label{f2003}
 p_n(A)=\|A+\mathcal J_n\|_{\mathfrak D/\mathcal J_n}
       =\max\left\{\|A_n\|_{\mathrm e},\,
                      \sup_{j\ne n}\|A_j\|\right\},
 \qquad A\in\mathfrak D,
\end{equation}
where $\|\cdot\|_{\mathrm e}$ denotes the essential norm on $\mathcal B(\mathcal H_n)$. The following lemma is a relative version of  \cite[Theorem1.1]{foi}

   \begin{lemma}\label{l2001}
Let $\mathfrak A\subseteq\mathfrak D$ be a unital norm-closed subalgebra. Suppose that
$\mathcal E_n(\mathfrak A)$ is transitive and $\mathfrak A\cap \mathcal J_n=\{0\}$. Then there exist
$x,y\in \mathcal H_n$ such that for every $A=\oplus_{n=1}^{\infty}A_n\in\mathfrak A$,
\begin{equation}\label{f2004}
 \langle x,y\rangle=1,
 \qquad
 |\langle A_nx,y\rangle |\leq p_n(A).
\end{equation}
\end{lemma}
\begin{proof}
Put $\mathcal S=\{A\in\mathfrak A:p_n(A)\leq1\}$.
We first show that there exists a nonzero vector $x_0$ for which
\[
 \overline{\{A_nx_0:A\in\mathcal S\}}\ne \mathcal H_n.
\]

Suppose first that for any nonzero $u\in\mathcal H_n$,
\begin{equation}\label{f2005}
 \overline{\{A_nu:A\in\mathcal S\}}=\mathcal H_n
\end{equation}
For every $\varepsilon>0$,  we claim that
\begin{equation}\label{f2006}
 \overline{\{A_nu:a\in\mathfrak A,\ p_n(a)<\varepsilon\}}=\mathcal H_n.
\end{equation}
 In fact, let $v\in \mathcal H_n$  and
$\eta>0$ . Put $t=\varepsilon/2$.
By (\ref{f2005}), there exists $C\in\mathcal S$ such that
\[
 \|C_nu-t^{-1}v\|<t^{-1}\eta.
\]
Set $A=tC\in\mathfrak A$. Then
\[
 p_n(A)=t\,p_n(C)\leq t<\varepsilon,
 \qquad
 \|A_nu-v\|
 =t\|C_nu-t^{-1}v\|<\eta.
\]
This proves (\ref{f2006}).

Next we choose a vector $u_0\in\mathcal H_n$ with $\|u_0\|=2$ and put $\Delta =\{u\in\mathcal H_n: \|u-u_0\|\leq1\}$. Then $\Delta $ is a weak compact convex subset.  Take any $\varepsilon\in (0,\frac18)$. For any $u\in\Delta$, there exists an $A(u)\in\mathfrak A$ such that
$$p_n(A(u))<\varepsilon \text { and } \|(A(u))_nu-u_0\|<\frac12$$ by (\ref{f2006})
Then there exists $K(u)\in \mathcal J_n$
such that
\[
 A(u)=K(u)+R(u),
 \qquad
 \|R(u)\|<\varepsilon.
\]
Since $(K(u))_n$ is compact, its restriction to the bounded set
$\Delta$ is continuous from the weak topology to the norm
topology. Hence $u$ has a relative weak neighborhood $V_u$ in
$\Delta$ such that
\[
 \|(K(u))_n(v-u)\|<\frac14, \ \ \forall v\in V_u.
\]
For $v\in V_u\subseteq \Delta $, we therefore have  $\|v-u\|\leq\|v-u_0\|+\|u_0-u\|\leq 2$ and
\begin{align*}
 \|(A(u))_{n}v-u_0\|
 &\leq
 \|(A(u))_{n}u-u_0\|
 +|(K(u))_{n}(v-u)\|
 +\|(R(u))_{n}(v-u)\|\\
 &< \frac12+\frac14+2\varepsilon
 <1.
\end{align*}
In particular, $(A(u))_{n}v\in \Delta$ whenever $v\in V_u$.
Choose a finite subcover $V_{u_1},\ldots,V_{u_m}$    and a continuous
partition of unity $\alpha_1,\ldots,\alpha_m$ subordinate to this
cover on the compact Hausdorff space $\Delta$(cf. \cite[Theorem 2.1.3]{rud}). Thus
$\alpha_i\geq0$, $\sum_i\alpha_i=1$, and
$\alpha_i(v)>0$ implies $v\in V_{u_i}$.
Define
\[
 F(v)=\sum_{i=1}^m\alpha_i(v)A(u_i)_{n}v,
 \qquad \forall v\in \Delta.
\]
This map is weakly continuous. Whenever $\alpha_i(v)>0$,
the corresponding vector $A(u_i)_{n}v$ belongs to $\Delta$.
By convexity, $F(\Delta)\subseteq\Delta$.
 By the Schauder-Tychonoff fixed point theorem(cf.\cite[Corollary 5.4.3]{ist}), there exists a
$v_0\in\Delta$ such that $F(v_0)=v_0$.
Set
\[
 A=\sum_{i=1}^m\alpha_i(v_0)A(u_i)\in\mathfrak A.
\]
Then
\[
 p_n(A)
 \leq\sum_{i=1}^m\alpha_i(v_0)p_n(A(u_i))
 <\varepsilon,
\]
and
\begin{equation}\label{f2007}
 A_nv_0=v_0,
 \qquad v_0\ne0.
\end{equation}
The last assertion follows from $0\notin\Delta$.

Choose $r$ such that
$
 p_n(A)<r<1.
$
Since $\|A_n\|_{\mathrm e}<r$, $\sigma(A_n)\cap\{\lambda \in\mathbb C:|\lambda |>r\}$ consists of
isolated points with finite-rank Riesz projections.

For the other coordinates,
\[
 \sup_{j\ne n}\|A_j\|\leq p_n(A)<r.
\]
This implies that $\sigma(A_j)\subseteq \{\lambda \in\mathbb C:|\lambda |\leq r\}$. Consequently,
  $$\sigma(A)\cap\{\lambda \in\mathbb C: |\lambda |>r\}=\sigma(A_n) \cap\{\lambda \in\mathbb C:|\lambda |>r\}.$$
By (\ref{f2007}), $1\in \sigma(A)\cap\{\lambda \in\mathbb C: |\lambda |>r\}=\sigma(A_n) \cap\{\lambda \in\mathbb C:|\lambda |>r\}$ is an isolated
point spectrum of $A$(resp. $A_n$). Then the Riesz idempotent $P_1$ corresponding to the spectral point $1$ of $A$(resp. $A_n$) is      nonzero and  has  finite rank. It is trivial that the range of $P_1$ is contained in $\mathcal H_n$ and hence $0\ne P_1\in \mathcal J_n$. This contradicts $\mathfrak A\cap \mathcal J_n=\{0\}$. Therefore there exists a nonzero vector
$x_0\in   \mathcal H_n$ such that
\[
 \mathcal C=\overline{\{A_nx_0:A\in\mathcal S\}}\ne   \mathcal H_n.
\]
 It is elementary that $\mathcal C$ is a   convex and balanced closed subset such that   $[\mathcal C]=\mathcal H_n$ since $\mathcal E_n(\mathfrak A)$ is transitive.
 Choose $w\notin \mathcal C$, let $x=P_Cw$ be the unique nearest point
of $\mathcal C$ to $w$, and put $y_0=w-x\ne0$.
By the projection theorem for closed convex sets in Hilbert spaces (see \cite[Theorem 5.2]{bre}), applied to the underlying real Hilbert space of \(\mathcal H\), we obtain
\[
 \operatorname{Re}\langle z,  y_0\rangle
 \leq\operatorname{Re}\langle x, y_0\rangle
 , \ \forall z\in \mathcal C.
\]
For every scalar $\omega$ with $|\omega|=1$, balance gives
$\omega z\in \mathcal C$. Maximizing over these scalars yields
\begin{equation}\label{f2008}
 |\langle z,y_0\rangle |
 \leq\operatorname{Re}\langle x, y_0\rangle
 , \forall z\in \mathcal C.
\end{equation}
Taking $z=x$ shows that $\langle x, y_0\rangle $ is real and nonnegative.
In fact, $ \langle x, y_0\rangle>0$. Otherwise $y_0\bot [\mathcal C]$, contradicting $y_0\ne0$.
For any  $A\in\mathcal S$ and $z\in \mathcal C$, choose $D_j\in\mathcal S$
such that $D_{j,n}x_0\to z$.
Since  $\|A_n\|\leq p_n(A)\leq 1$, $AD_j\in\mathcal S$,  and
\[
 A_nz=\lim_j(AD_j)_nx_0\in \mathcal C.
\]
It follows that $A_n\mathcal C\subseteq \mathcal C$. In particular, $A_nx\in \mathcal C$, and
(\ref{f2008}) gives
\[
 |\langle A_nx, y_0\rangle |\leq  \langle x, y_0\rangle
 , \ \forall A\in\mathcal S.
\]
For arbitrary $A\in\mathfrak A$ with $p_n(A)>0$, apply this inequality
to $\frac{A}{p_n(A)}$ to obtain
\[
 |\langle A_nx, y_0\rangle |\leq \langle x,y_0\rangle p_n(A).
\]
If $p_n(A)=0$, then $A\in\mathfrak A\cap \mathcal J_n$ and $A=0$, so the same
inequality holds.
Finally, Putting  $y=\frac{y_0}{\langle x,y_0\rangle}$, we have
\[
 \langle x, y\rangle =1,
 \qquad
 |\langle A_nx, y\rangle |\leq p_n(A)
 , \ \forall A\in\mathfrak A.
\]
 \end{proof}

  \begin{lemma}\label{l22}
  Under the hypotheses of Lemma \ref{l2001}, there are a
Hilbert space $\mathfrak H$, a unital $*$-representation  $\rho$  of $\mathfrak D$ on $\mathfrak H$ with
  $\rho(\mathcal J_n)=0$,    a  closed subspace $\mathfrak M\subset \mathfrak H$ which is an  invariant subspace of    $\rho(\mathfrak A)$, and
   a closed densely defined  injective operator $T$  from the domain $\mathcal D(T)\subseteq \mathcal H_n$ to
$\mathfrak M^{\bot}$
such that, for every $A\in\mathfrak A$,
\begin{equation}\label{f2009}
 A_n\mathcal D( T)\subseteq\mathcal D( T),
 \qquad
 TA_nu=P_{\mathfrak M^\perp}\rho(A)Tu,\ \forall u\in\mathcal D( T).
\end{equation}
\end{lemma}
\begin{proof}
Let $\pi_n:\mathfrak D\to\mathfrak D/\mathcal J_n$ be the quotient map. By
Lemma \ref{l2001}, we  may define
\[
 \varphi(\pi_n(A))=\langle A_nx,y \rangle ,\qquad \forall A\in\mathfrak A.
\]
Then $\varphi$ is an well-defined bounded linear functional on $\pi_n(\mathfrak A)$.
with $\|\varphi\|=\varphi(I)=1$.  By the Hahn-Banach theorem, $\varphi$ may be extended  to a bounded linear functional which is still denoted by  $\varphi$ on $\mathfrak D/\mathcal J_n$. Such a functional on a unital $C^*$-algebra
is a state. Its GNS representation, composed with $\pi_n$, gives a
unital $*$-representation $\rho:\mathfrak D\to B(\mathfrak H)$ and a unit vector
$\xi\in\mathfrak H$ such that $\rho (\mathcal J_n)=0$ and
\begin{equation}\label{f2010}
 \langle A_nx, y\rangle =\langle\rho(A)\xi,\xi\rangle ,\ \forall A\in\mathfrak A.
\end{equation}

 Let  $L$  be the closed subspace of  $  \mathcal H_n\oplus \mathfrak H$  spanned  by   $\{A_nx\oplus \rho(A)\xi: A\in\mathfrak A\}.$ It is known  that
 $L$ is invariant under $A_n\oplus \rho(A)$ for any $A\in\mathfrak A$ since $\mathfrak A$ is an algebra.  On the other hand,
 we have
 \begin{equation}\label{f2122}\langle A_nx\oplus \rho(A)\xi, y\oplus- \xi\rangle=\langle A_nx,y\rangle-\langle \rho(A)\xi,\xi\rangle =0\end{equation}
  by (\ref{f2010}). This  implies that  $L\not=\mathcal H\oplus \mathfrak H$. Put
 $\mathfrak M=\{\eta\in\mathfrak H: 0\oplus \eta\in L\}$. Then $\rho(\mathfrak A)\mathfrak M\subseteq \mathfrak M$.  For any $u\oplus v\in L$, we have
 \begin{equation}\label{f212} u\oplus P_{\mathfrak M^{\bot}}v=u\oplus v-0\oplus P_{\mathfrak M}v\in L.
 \end{equation}
  This means that
 $$G=L\cap( \mathcal H_n\oplus \mathfrak M^{\bot})$$
 is a graph of a closed operator $T$ with domain $\mathcal D(T)=\{u\in\mathcal H_n: u\oplus v\in L\}$ which contains $\{A_nx:A\in\mathfrak A\}$. Thus $T$ is densely defined.

  Since $\ker T$ is invariant under $\mathcal E_n(\mathfrak A)$, $\ker T=\{0\}$. Otherwise,  since $\mathcal E_n(\mathfrak A)$ is transitive,  $T=0$.  This implies that   $\mathcal H_n\oplus\{0\}\subseteq L$. This contradicts with (\ref{f2122}) since $y\not=0$. Thus  $T$ is injective. For any    $u\in\mathcal D(T)$ and $A\in\mathfrak A$ , we have   $u\oplus  Tu\in L$ and thus   $A_nu\oplus \rho(A)Tu\in L$ since $L$ is invariant under $A_n\oplus \rho(A)$ for every $A\in\mathfrak A$. By (\ref{f212}), $A_nu\oplus P_{\mathfrak M^{\bot}}\rho(A)Tu\in L$. Hence $A_nu\in\mathcal D(T)$ and $TA_nu =P_{\mathfrak M^{\bot}}\rho(A)Tu$, that is, (\ref{f2009}) holds.
\end{proof}

We recall that a  densely defined closed operator $Q$ with domain $\mathcal D(Q)$  in a separable infinite dimensional Hilbert sapce $\mathcal H$   is   selfadjoint if $Q=Q^*$.  If   $\langle Qx,x\rangle\geq0$ for any $x\in\mathcal D(Q)$, then $Q$ is said  to be  positive selfadjoint(cf.\cite[$p.{316}$, Exercise 9]{co}).  Let $P\in\mathcal B(\mathcal H)$ be  a  positive invertible operator. Then $Q_P=P^{-\frac{1}2}QP^{-\frac{1}2}$ is   positive  self-adjoint with domain $\mathcal D(Q_P)=\{x\in\mathcal H: P^{-\frac{1}2}x\in \mathcal D(Q)\}$. We denote by $\sigma_p(Q)$ the point spectra of $Q$.

\begin{lemma}\label{l21}  Let $Q$ be an    injective positive  selfadjoint operator on $\mathcal H$
 and  \(F\subseteq\mathbb R\)   a countable subset. Then
for every \(\delta>0\), there exists a positive bounded invertible
operator \(P\in\mathcal B(\mathcal H)\) such that
\begin{equation}\label{f21}
 P-I\in\mathcal K(\mathcal H),\qquad \|P-I\|<\delta,
\end{equation}
and \(Q_P \) has a complete orthonormal basis of
eigenvectors with
\[
 \sigma_p(Q_P)\cap F=\emptyset.
\]
\end{lemma}
\begin{proof}
Let $Q=\int_{[0,\infty)}\lambda dE_{\lambda}$ be the  spectral decomposition  of $Q$.  For every $n\in\mathbb Z$, put   $\Delta_n=[2^n,2^{n+1})$, and $H_n=E(\Delta_n)\mathcal H$. Since $Q$ is injective,
we have
\begin{equation}\label{f22}
\mathcal H=\oplus_{n=-\infty}^{+\infty} H_n \text{  and }Q=\oplus_{n=-\infty}^{+\infty} Q_n, \end{equation}
where $Q_n=Q|_{H_n}$  is  bounded and positive invertible for every $n\in\mathbb Z$. In fact,
\begin{equation}\label{fr23}
2^n\leq Q_n\leq  2^{n+1}, \ \ \ \ \forall n\in \mathbb Z.
\end{equation}
 We  assume that $\delta<1$ and fixed any  $n\in\mathbb Z$.  Put $\varepsilon_n=\frac{\delta }{4^{|n|+1}}$.  Then $\frac{\varepsilon_n}{1-\varepsilon_n}<\frac{\delta}{3}$.
 By the Weyl-von Neumann Theorem(cf. \cite[$\S$38]{con}), there exists a diagonal positive operator $C_n\in\mathcal B(H_n)$ such that $Q_n-C_n\in\mathcal K(  H_n)$ and $\|Q_n-C_n\|<2^{n-1}\varepsilon_n $. If $\sigma_p(C_n)\cap F=\emptyset$, then we take $D_n=C_n$. Otherwise, we may take a diagonal operator $D_n$  which has  the same diagonals with $C_n$  such that

  $\sigma_p(D_n)\cap F=\emptyset,$ $ C_n-D_n\in\mathcal K(H_n)$     and $ \|C_n-D_n\|< 2^{n-1}\varepsilon_n $.
 Thus \begin{equation}\label{f24}
 \sigma_p(D_n)\cap F=\emptyset,\ \  Q_n-D_n\in\mathcal K(H_n) \text{ and } \|Q_n-D_n\|<  2^n \varepsilon_n. \end{equation} In particular,
 \begin{equation}\label{f25}
Q_n - \varepsilon_n 2^n<D_n<Q_n+  2^n\varepsilon_n.\end{equation}
 By (\ref{fr23}) and (\ref{f25}), we have
 \begin{equation}\label{f26}
 (1-\varepsilon_n)Q_n<D_n<(1+\varepsilon_n)Q_n.
 \end{equation} It follows that $D_n$ is positive invertible. Put $P_n=Q_n^{\frac12}D_n^{-1}Q_n^{\frac12}$. Then
 \begin{equation}\label{f27}
 (1+\varepsilon_n)^{-1}<P_n<(1-\varepsilon_n)^{-1},\ \ \  \|P_n-I\|<\frac{\varepsilon_n}{1-\varepsilon_n}<\frac{\delta}{3}.
 \end{equation}
 Furthermore, since $D_n^{-1}-Q_n^{-1}=D_n^{-1}(Q_n-D_n)Q_n^{-1}\in\mathcal K(H_n)$, $$P_n-I=Q_n^{\frac12}\left(D_n^{-1}-Q_n^{-1}\right)Q_n^{\frac12}\in\mathcal K(H_n).$$
  Let $P=\oplus_{n=-\infty}^{+\infty}P_n$. Then $P\in\mathcal B(\mathcal H)$ is a positive invertible operator such that $P-I\in \mathcal K(\mathcal H)$ and  $\|P-I\|<\frac{\delta}3<\delta$ by (\ref{f27}).

  Let $T_n=Q_n^{\frac12}P_n^{-\frac12}=V_n |T_n|$  be the polar decomposition . We have $V_n$ is unitary and
  \begin{align}  \label{f28}
 & |T_n|^2=T_n^*T_n=P_n^{-\frac12}Q_nP_n^{-\frac12},\\
 &T_nT_n^*=V_n \left( P_n^{-\frac12}Q_nP_n^{-\frac12}\right)V_n^*=Q_n^{\frac12}P_n^{-1}Q_n^{\frac12}=D_n.
  \end{align}
   Note that $Q_P=P^{-\frac12}QP^{-\frac12}=\oplus_{n=-\infty}^{+\infty}P_n^{-\frac12}Q_nP_n^{-\frac12} =\oplus_{n=-\infty}^{+\infty}T_nT_n^*$ with domain
   $$\mathcal D(Q_P)=\{\oplus_{n=-\infty}^{+\infty}x_n: \sum_{n=-\infty}^{+\infty}\|P_n^{-\frac12}Q_nP_n^{-\frac12}x_n\|^2<\infty\}.$$
   Put $V=\oplus_{n=-\infty}^{+\infty}V_n$. Then we have
   $$VQ_PV^*=\oplus_{n=-\infty}^{+\infty}D_n=D=\oplus_{n=-\infty}^{+\infty}T_n^*T_n.$$ It follows that $\sigma_p(Q_P)=\sigma_p(D)$. Hence $\sigma_p(Q_P)\cap F=\emptyset$.
\end{proof}

Let $T$ be the closed operator defined in Lemma \ref{l22} and put  $Q=T^*T$. Then $Q$ is  positive self-adjoint  and injective.
\begin{lemma}\label{l23} Let $\mathfrak A$ and $Q$  be as above. Let $P\in\mathfrak D$ be  positive invertible such that $P-I\in\mathcal J_n$.  Suppose that  $A\in\mathfrak A$  and $PA\in\mathfrak A^*$. Then  we have
\begin{equation}\label{f214}
A_n\mathcal D(Q)\subseteq \mathcal D(Q) \text{ and } QA_nu=P_nA_nQu, \forall u\in\mathcal D(Q).
\end{equation}Furthermore,
\begin{equation}\label{f215}
P_n^{\frac12}A_n \mathcal D(Q)\subseteq \mathcal D(Q_{P_n}) \text{ and } Q_{P_n}P_n^{\frac12}A_nu=P_n^{\frac12}A_nQu, \forall u\in\mathcal D(Q)
\end{equation} and
\begin{equation}\label{f216}
A_n^*P_n^{\frac12} \mathcal D(Q_{P_n})\subseteq \mathcal D(Q) \text{ and } QA_n^*P_n^{\frac12}v=A_n^*P_n^{\frac12}Q_{P_n}v, \forall v\in\mathcal D(Q_{P_n}).
\end{equation}
\end{lemma}
\begin{proof}
Put $B=(PA)^*=A^*P\in \mathfrak A$. Since $\rho(\mathcal J_n)=\{0\}$, $\rho(P)=I$ and
  $\rho(B)=\rho(A)^*\in\rho(\mathfrak A).$
Note that $\mathfrak M$ is invariant under $\rho(\mathfrak A)$ by Lemma \ref{l22}. It is a reductive subspace of $\rho(A)$. Put $W=\rho(A)|_{\mathfrak M^{\bot}}$.  By equation (\ref{f2009}), we have
\begin{equation}\label{f217}
TA_nu=WTu, \ \ TB_nu=W^*Tu,\forall u\in\mathcal D(T).
\end{equation}Take any $u\in\mathcal D(Q)\subseteq\mathcal D(T)$ and $v\in \mathcal D(T)$, we have
\begin{align*}
 &\ \ \ \  \langle TA_nu,Tv\rangle =\langle WTu, Tv\rangle =\langle Tu, W^*Tv\rangle \\
&=\langle Tu, TB_nv\rangle =\langle T^*Tu, B_nv\rangle=\langle Qu,B_nv\rangle =\langle B_n^*Qu, v\rangle=\langle P_nA_Qu,v\rangle
  \end{align*} by Lemma \ref{l22} and (\ref{f217}).
This means that $TA_nu\in\mathcal D(T^*)$   from the definition of $\mathcal D(T^*)$ and $T^*TA_nu=P_nA_nQu$, that is, $A_nu\in \mathcal D(Q)$ and $QA_nu=P_nA_nQu$. Hence,  (\ref{f214}) holds.

 Furthermore, for any $u\in\mathcal D(Q)$, $P_n^{-\frac12}P_n^{\frac12}A_nu=A_nu\in\mathcal D(Q)$.
  It follows that $QA_nu=Q P_n^{-\frac12}P_n^{\frac12}A_nu$.  Hence
  $P_n^{-\frac12}QA_nu=(P_n^{-\frac12}Q P_n^{-\frac12})P_n^{\frac12}A_nu$. This implies that  $P_n^{\frac12}A_nu\in\mathcal D(Q_{P_n})$ and
  $$Q_{p_n}P_n^{\frac12}A_nu=P_n^{-\frac12}QA_nu=P_n^{-\frac12}P_nA_nQu=P_n^{\frac12}A_nQu$$  by (\ref{f214}). This proves   (\ref{f215}).

  On the other hand, Note that $Q_{P_n}$ is self-adjoint.  For any $v\in \mathcal D(Q_{P_n})$ and $u\in \mathcal D(Q)$,
  $$\langle Qu, A_n^*P_n^{\frac12}v\rangle=\langle P_n^{\frac12}A_nQu,v\rangle=\langle Q_{p_n}P_n^{\frac12}A_nu, v\rangle=\langle u, A_n^*P_n^{\frac12}Q_{P_n}v\rangle
  $$ by (\ref{f215}).
  This implies that $A_n^*P_n^{\frac12}v\in \mathcal D(Q)$ since $Q=Q^*$. We again have $$QA_n^*P_n^{\frac12}v= A_n^*P_n^{\frac12}Q_{P_n}v,$$ that is, (\ref{f216}) holds.
\end{proof}
\begin{lemma}\label{l24} Fix n. Let $\mathfrak{A}\subseteq\mathfrak D$ be  a   unital   norm   closed subalgebra  such that
$\mathcal E_n(\mathfrak A)$ is transitive and $\mathfrak A\cap \mathcal J_n=\{0\}$. Then for any $\delta>0$,  there exists a positive invertible operator $P\in\mathfrak D$ satisfying
\begin{equation}\label{f218}
  P-I\in\mathcal J_n, \ \ \|P-I\|<\delta ,
   \end{equation}
    such that
    \begin{equation}\label{f2118}
    A\in \mathfrak A, \ \ PA\in\mathfrak A^*\ \  \Longrightarrow A_n=0.
    \end{equation}
\end{lemma}
\begin{proof} Let $T$ and $Q=T^*T$ be  closed operators on $\mathcal H_n $ as in Lemmas \ref{l22} and \ref{l23}.  Note that $\sigma_p(Q)$ is at most countable. Put $F=\sigma_p(Q)$. Then there exists a positive invertible operator $P_n$  in $\mathcal B(\mathcal H_n)$ with $P_n-I\in\mathcal K(\mathcal H_n)$ and $\|P_n-I\|<\delta$,  and an orthogonal basis $\{e_k:k\geq 1\}$ of $\mathcal H_n$ such that $\sigma_p(Q_{P_n})\cap F=\emptyset $ and
$Q_{P_n}e_k=\lambda_n e_k$, $\lambda_k\in \sigma_p(Q_{P_n})$ by  Lemma \ref{l21}.  In particular $\{e_k:k\geq1\}\subset\mathcal D(Q_{P_n})$.  Let $P\in\mathfrak D$ have $n$th coordinate $P_n$ and all other coordinates
equal to the identity.
Suppose that  $A\in\mathfrak A$  and $PA\in\mathfrak A^*$. By Lemma \ref{l23},
  $A_n^*P_n^{\frac12}e_k\in \mathcal D(Q)$ and  $QA_n^*P_n^{\frac12}e_k=A_n^*P_n^{\frac12}Q_{P_n}e_k=\lambda_kA_n^*P_n^{\frac12}e_k$, for every $k\geq1$. Since $F\cap \sigma_p(Q_{P_n})=\emptyset$,
  $A_n^*P_n^{\frac12}e_k=0$ for every $k$. Thus $A_n=0$, proving (\ref{f2118}).
\end{proof}
We next give a positive answer for the second part of Question \ref{q1}.
\begin{theorem}\label{t1}
Let $\mathfrak{A}\subseteq\mathcal{B}(\mathcal{H})$   be  a   unital   weak operator topology  closed transitive subalgebra of $\mathcal B(\mathcal H)$. If $\mathfrak{A}$ has the left (resp.\ right) partial  factorization property, then $\mathfrak{A}=\mathcal{B}(\mathcal{H})$.
\end{theorem}
\begin{proof} We take one Hilbert space in Lemma \ref{l24}.
  Suppose $\mathfrak{A}\cap \mathcal K(\mathcal H)=\{0\}$.  By Lemma \ref{l24},  for any $\delta>0$,   there exists a positive invertible operator $P$ such that (\ref{f218}) holds.  Let $S=P^{\frac12}$. Then
  $\pi(S)=I$ and $S-I=(P-I)(S+I)^{-\frac12}\in\mathcal K(\mathcal H)$. By the assumption, there exists an isometry(resp. a co-isometry) $U\in\mathcal B(\mathcal H)$ such that $ U^*S, S^{-1}U \in\mathfrak A$.
  Now, we have
  $P(S^{-1}U)=SU=(U^*S)^*$, that is $P(S^{-1}U)\in P\mathfrak A\cap \mathfrak A^*$.   This is a contradiction. Therefore $\mathfrak A\cap \mathcal K(\mathcal H)\ne\{0\}$. It follows that $\mathfrak A=\mathcal B(\mathcal H)$ from \cite[Theorem 8.23]{rad}.
 \end{proof}
  \begin{corollary}\label{c2}
Let $\mathfrak{A}\subsetneq\mathcal{B}(\mathcal{H})$ be a   unital   weak operator topology  closed  subalgebra of $\mathcal B(\mathcal H)$. If $\mathfrak{A}$ has  the left (resp.\ right) partial factorization property, then $\mathfrak{A}$ admits a nontrivial invariant subspace.
\end{corollary}

To consider a  positive answer for the first part of Question \ref{q1}, we need the following theorem.
\begin{lemma} \label{l2006}
Let $\mathfrak A\subseteq\mathfrak D$ be weak operator topology  closed  such that
  $\mathcal E_n(\mathfrak A)$ is  dense in $\mathcal B( \mathcal H_n)$ in weak operator topology
   at a fixed coordinate $n$.
If $\mathfrak A\cap E_n\mathfrak D\ne\{0\}$, then $E_n\mathfrak D\subseteq\mathfrak A$.
In particular, $\mathfrak A\cap \mathcal J_n\ne\{0\}$ implies $E_n\in\mathfrak A$.
\end{lemma}
\begin{proof}
Identify $E_n\mathfrak D$ with $B(\mathcal \mathcal \mathcal H_n)$ and put $\mathcal I_n=\mathfrak A\cap E_n\mathfrak D$.
This is a weak operator topology  closed subalgebra. If $K\in\mathcal I_n$
and $X\in \mathcal B(  \mathcal H_n)$, choose a net $A_{\alpha}\in\mathfrak A$ such that
$A_{\alpha, n}\to X$ in the weak operator topology.
Since $E_nK=K=KE_n$, separate weak-operator continuity
of multiplication gives
\[
 A_{\alpha} K\longrightarrow XK,\qquad
 KA_{\alpha})\longrightarrow KX
\]
in the weak operator topology on $\mathcal B(\mathcal H)$ such that $E_nXK=XK=XKE_n$ and $E_nKX=KX=KXE_n$. Hence $XK,KX\in\mathcal I_n$.
Thus $\mathcal I_n$ is a nonzero weak-operator closed two-sided ideal
in $\mathcal B(  \mathcal H_n)$.  Therefore
$\mathcal I_n=\mathcal B(  \mathcal H_n)$.
\end{proof}
\begin{theorem}\label{t2}
Let $\mathfrak A\subseteq\mathfrak D$ be weak operator topology  closed  such that
  $\mathcal E_n(\mathfrak A)$ is  dense in $\mathcal B( \mathcal H_n)$ in weak operator topology
    for every $n$.
If $\mathfrak A$ has   the left(resp.  right) partial factorization property
in $\mathfrak D$, then $E_n\in\mathfrak A$ for every $n$ and $\mathfrak A=\mathfrak D$.
\end{theorem}
\begin{proof}
   Suppose that $\mathfrak A\cap \mathcal J_n=\{0\}$ for some $n$. Note that $\mathcal E_n(\mathfrak A)$ is transitive.
  Apply Lemma \ref{l24} and set $S=P^{1/2}$.
By the  assumption,   there is an
isometry(resp. co-isometry) $U\in\mathfrak D$ such that
\[
  A =S^{-1}U,B=U^*S\in\mathfrak A.
\]
As $S$ is positive,
\[
 PA=S^2S^{-1}U=SU=(U^*S)^*=B^*\in\mathfrak A^*.
\]
Then $A_n=0$ by Lemma \ref{l24}, so $U_n=0$.
This is impossible: $U^*U=I$ implies $U_n^*U_n=I_{  \mathcal H_n}$, and
$UU^*=I$ implies $U_nU_n^*=I_{  \mathcal H_n}$. In either case $U_n\ne0$.
Therefore $\mathfrak A\cap \mathcal J_n\ne\{0\}$.
Lemma \ref{l2006} now gives $E_n\mathfrak D\subseteq\mathfrak A$ for every
$n$.  Hence $\mathfrak A=\mathfrak D$. \end{proof}

\section{Reflexivity of operator algebras  with  the   partial factorization property}
 We recall that a subalgebra $\mathcal{A}$ of $\mathcal{B}(\mathcal{H})$ is reflexive if
\[
\mathcal{A} = \operatorname{Alg}\operatorname{Lat}\mathcal{A}
= \bigl\{A\in\mathcal{B}(\mathcal{H}) : E^{\perp}AE = 0,\ \forall E\in \operatorname{Lat}\mathcal{A}\bigr\}.
\]
  A lattice $\mathcal N$ in $\mathcal P(\mathcal B(\mathcal H))$ is a nest if it is completely ordered. For any nonzero $N\in\mathcal{N}$, we define
\[
N_-=\bigvee\bigl\{P\in\mathcal{N}: P<N\bigr\}.
\]
The projections $N\ominus N_-$ are called the atoms of $\mathcal{N}$. If
\[
\bigoplus_{N\in\mathcal{N}} N\ominus N_-=I,
\]
then $\mathcal{N}$ is called atomic.

   We next  assume that $\mathfrak A$ is a weak operator topology closed subalgebra  with the left(resp.\ right) partial factorization property. Then $\operatorname{Lat}\mathfrak A$ is a nest by \cite[Corollary 2.9]{ji} and
  $\operatorname{Alg}\operatorname{Lat}\mathfrak A  $ is  the nest algebra associated with $\operatorname{Lat}\mathfrak A$. Define the diagonal algebra by
\[
\mathfrak{D} = \operatorname{Alg}\operatorname{Lat}\mathfrak A \cap (\operatorname{Alg}\operatorname{Lat}\mathfrak A)^*.
\]
Then $\mathfrak{A}\subseteq\operatorname{Alg}\operatorname{Lat}\mathfrak A$.

\begin{lemma}\label{l1}
Let $T\in\operatorname{Alg}\operatorname{Lat}\mathfrak A$ be invertible with $T^{-1}\in\operatorname{Alg}\operatorname{Lat}\mathfrak A$. Then there exists an isometry (resp.\ co-isometry) $U\in\mathfrak{D}$ such that $U^*T\in\mathfrak{A}$ and $T^{-1}U\in\mathfrak{A}$. In particular, $\mathfrak{A}\cap\mathfrak{D}$ has the left (resp.\ right) partial factorization property within $\mathfrak{D}$.
\end{lemma}
\begin{proof} Since $\mathfrak{A}$ has the left (resp.\ right) partial factorization property, there exists an isometry (resp.\ co-isometry) $U\in\mathfrak{D}$ satisfying $U^*T\in\mathfrak{A}$ and $T^{-1}U\in\mathfrak{A}$. It follows that
\[
U^* = U^*TT^{-1}\in\operatorname{Alg}\operatorname{Lat}\mathfrak A,\qquad U = TT^{-1}U\in\operatorname{Alg}\operatorname{Lat}\mathfrak A,
\]
which implies $U\in\mathfrak{D}$. In particular, if $T\in\mathfrak{D}$, then $U^*T\in\mathfrak{D}\cap\mathfrak{A}$ and $T^{-1}U\in\mathfrak{D}\cap\mathfrak{A}$.
\end{proof}
  It is elementary that $\operatorname{Alg}\operatorname{Lat}\mathfrak{A}$ has the left (resp.\ right) partial factorization property, owing to the inclusion $\mathfrak{A} \subseteq \operatorname{Alg}\operatorname{Lat}\mathfrak{A}$. By \cite[Theorems~1 and~2]{jigs}, $\operatorname{Lat}\mathfrak{A}$ is an atomic (resp.\ at most countable) nest. Note that every at most countable nest is atomic. In both cases, we let $\{E_{\lambda} : \lambda \in \Lambda\}$ denote the atoms of $\operatorname{Lat}\mathfrak{A}$, where $\Lambda$ is an at most countable index set. It is well known that
\begin{equation}\label{f1}
\mathfrak D=\bigoplus_{\lambda\in\Lambda}E_{\lambda}\mathcal B(\mathcal H)E_{\lambda}=\bigoplus_{\lambda\in\Lambda}\mathcal B(E_{\lambda} \mathcal H)
\end{equation}
and  the center $\mathcal Z(\mathfrak D)$ of $\mathfrak D$ is given by $\{E_{\lambda}:\lambda\in\Lambda\}''$. If we may define
 $$\Phi(A)=\sum\limits_{\lambda\in\Lambda}E_{\lambda}AE_{\lambda}$$
  for all $A\in\mathcal B(\mathcal H)$, then $\Phi$ is  a unique faithful normal conditional expectation   from $\mathcal B(\mathcal H)$ onto $\mathfrak D$ such that  $\tau\circ \Phi=\tau$  and $\operatorname{Alg}\operatorname{Lat}\mathfrak A$ is a maximal subdiagonal algebra of $\mathcal B(\mathcal H)$ with respect to $\Phi$, where $\tau $ is the unique faithful normal semi-finite trace on $\mathcal B(\mathcal H)$(cf.\cite[Corollary 3.1.2]{arv1} and \cite[Theorem 3.1]{jio}). Put $(\operatorname{Alg}\operatorname{Lat}\mathfrak A)_0=\{A\in \operatorname{Alg}\operatorname{Lat}\mathfrak A:\Phi(A)=0\}$. Then $ (\operatorname{Alg}\operatorname{Lat}\mathfrak A)_0$ is a two side ideal of $\operatorname{Alg}\operatorname{Lat}\mathfrak A$ such that $\operatorname{Alg}\operatorname{Lat}\mathfrak A=\mathfrak D+(\operatorname{Alg}\operatorname{Lat}\mathfrak A)_0$.  Moreover, for each $\alpha\in\Lambda$, $E_{\alpha}=P_{\alpha}\ominus (P_{\alpha})_-$ for a unique $P_{\alpha}\in\operatorname{Lat}\mathfrak{A}$. We may thus define an order on $\Lambda$ by setting $\alpha<\beta$ whenever $P_{\alpha}<P_{\beta}$. Then $E_{\alpha}(\operatorname{Alg}\operatorname{Lat}\mathfrak A)E_{\beta}=\{0\}$ for any $\alpha >\beta$ and $\operatorname{Alg}\operatorname{Lat}\mathfrak A=\bigvee_{\alpha\leq \beta}E_{\alpha}\mathcal B(\mathcal H)E_{\beta}$   in the ($\sigma$-)weak operator topology.
\begin{lemma}\label{l2}
 For any $\lambda\in\Lambda$, the algebra $E_{\lambda}\mathfrak{A}E_{\lambda}$ is dense in $\mathcal{B}(E_{\lambda}\mathcal{H})$ with respect to the weak operator topology.
\end{lemma}
\begin{proof}
Note that \( E_{\lambda} = P_{\lambda} \ominus (P_{\lambda})_- \) for some \( P_{\lambda} \in \operatorname{Lat}\mathfrak A \), and \( E_{\lambda} \) is a semi-invariant subspace of \( \mathfrak{A} \). It is well known that \( E_{\lambda}\mathfrak{A}E_{\lambda} \) is an algebra. Suppose, for contradiction, that there exists some \( \lambda \) such that \( E_{\lambda}\mathfrak{A}E_{\lambda} \) is not dense in \( \mathcal{B}(E_{\lambda}\mathcal{H}) \) with respect to the weak operator topology. Then, for any invertible operator \( T \in \mathcal{B}(E_{\lambda}\mathcal{H}) \), the operator \( S = T \oplus I \) is invertible in \( \mathfrak{D} \). Consequently, there exists an isometry (resp. co-isometry) \( U = U_1 \oplus U_2 \in \mathfrak{D} \) such that \( U^*S \) and \( S^{-1}U \) both lie in \( \mathfrak{D} \cap \mathfrak{A} \) by Lemma \ref{l1}. This implies that \( U_1^*T \) and \( T^{-1}U_1 \) are in \( E_{\lambda}\mathfrak{A}E_{\lambda} \). It is trivial that \( U_1 \) is an isometry (resp. co-isometry), so \( E_{\lambda}\mathfrak{A}E_{\lambda} \) has the left (resp. right) partial factorization property.

Hence, by Corollary \ref{c2}, \( E_{\lambda}\mathfrak{A}E_{\lambda} \) admits a nontrivial invariant subspace \( Q \) in \( \mathcal{B}(E_{\lambda}\mathcal{H}) \). It follows that \( (P_{\lambda})_- \oplus Q \) is an invariant subspace of \( \mathfrak{A} \) satisfying \( (P_{\lambda})_- < (P_{\lambda})_- \oplus Q < P_{\lambda} \). In fact, \( \left((P_{\lambda})_-\oplus Q\right)^{\perp}=(E_{\lambda}\ominus Q)\oplus P_{\lambda}^{\perp}\). Thus
\begin{align*}
&\ \ \ \ \bigl((P_{\lambda})_{-} \oplus Q\bigr)^{\perp} A \bigl((P_{\lambda})_{-} \oplus Q\bigr)\\
&= \bigl((E_{\lambda} \ominus Q) \oplus P_{\lambda}^{\perp}\bigr) A \bigl((P_{\lambda})_{-} + Q\bigr) \\
&= \bigl((E_{\lambda} \ominus Q) \oplus P_{\lambda}^{\perp}\bigr) (P_{\lambda})_{-}^{\perp} A (P_{\lambda})_{-}
+ \bigl((E_{\lambda} \ominus Q) \oplus P_{\lambda}^{\perp}\bigr) A Q P_{\lambda} \\
&= (E_{\lambda} \ominus Q) A Q \\
&= (E_{\lambda} \ominus Q) E_{\lambda} A E_{\lambda} Q = 0
\end{align*}
 for all $A\in \operatorname{Alg}\operatorname{Lat}\mathfrak A$.
This contradicts the fact that \( E_{\lambda} \) is an atom of \( \operatorname{Lat}\mathfrak A \). Therefore, \( E_{\lambda}\mathfrak{A}E_{\lambda} \) must be dense in \( \mathcal{B}(E_{\lambda}\mathcal{H}) \) with respect to the weak operator topology.
\end{proof}
 \begin{theorem}\label{t3}
 Let $\mathfrak A$ be a unital  closed subalgebra of $\mathcal B(\mathcal H)$  in weak operator topology. If $\mathfrak A$ has the left(resp.\ right) partial factorization property, then $\mathfrak A=\operatorname{Alg}\operatorname{Lat}\mathfrak A$. In particular, $\mathfrak A$ is reflexive.
\end{theorem}
\begin{proof} By Lemma \ref{l1}, $\mathfrak A\cap \mathfrak D$ has the left(resp.\ right) partial factorization property in $\mathfrak D$. By Lemma \ref{l2} and Theorem \ref{t2},
\begin{equation}\label{ff}\mathfrak A\cap\mathfrak A^*=\mathfrak A\cap \mathfrak D=\mathfrak D =\bigoplus\limits_{\lambda\in\Lambda} \mathcal B(E_{\lambda}\mathcal H)\end{equation}
  since $\Lambda$ is at most countable.
If $\mathfrak A$ has the right partial factorization property, then for any invertible element $T\in\operatorname{Alg}\operatorname{Lat}\mathfrak A$ with $T^{-1}\in\operatorname{Alg}\operatorname{Lat}\mathfrak A$, there exists a co-isometry $U\in \mathfrak D$ such that $U^*T$ and $T^{-1} U$ belong to $\mathfrak A$  By Lemma \ref{l1}. Then $T=UU^*T\in \mathfrak A$ by (\ref{ff}). Thus $\mathfrak A=\operatorname{Alg}\operatorname{Lat}\mathfrak A$.

Next, we assume that $\mathfrak A$ has the left partial factorization property.
 by   (\ref{ff}),  $\mathfrak A=\mathfrak D+\mathfrak A_0$, where $\mathfrak A_0=\{A\in\mathfrak A: \Phi(A)=0\}=\mathfrak A\cap (\operatorname{Alg}\operatorname{Lat}\mathfrak A)_0$. It suffices to show that $\mathfrak A_0=(\operatorname{Alg}\operatorname{Lat}\mathfrak A)_0$. If $\mathfrak A_0\neq(\operatorname{Alg}\operatorname{Lat}\mathfrak A)_0$, then there exists a rank-1 operator $x\otimes y\in (\operatorname{Alg}\operatorname{Lat}\mathfrak A)_0 \setminus \mathfrak A_0$ by \cite[Corollary 3.13]{dav2}.   By \cite[Lemma 3.1]{dav2}, there exists a projection $P_{t}\in \operatorname{Lat}\mathfrak A$ such that $x\in P_t(\mathcal H)$ and $y\in (P_t)_-^{\perp}(\mathcal H)$. Since  $\Phi(x\otimes y)=0$ and $E_{\lambda}\mathcal B(\mathcal H)E_{\lambda}\subseteq \mathfrak D\subseteq \mathfrak A$, we easily have $y\in P_t^{\perp}(\mathcal H)$. Let
$$T=I+x\otimes y\in\operatorname{Alg}\operatorname{Lat}\mathfrak A.$$
Then $T^{-1}=I-x\otimes y\in\operatorname{Alg}\operatorname{Lat}\mathfrak A$. Thus there exists an isometry $U\in \mathfrak D$ such that $U^*T$ and $T^{-1}U$ are elements of $\mathfrak A$ by Lemma \ref{l1}. It follows that $U^*x\otimes y\in \mathfrak A_0$ and $(x\otimes y)U\in\mathfrak A_0$. Hence $UU^*x\otimes y\in\mathfrak A$. Define
$$Q=\sup\big\{P\in\mathfrak D: P \text{ is a projection and } Px\otimes y\in \mathfrak A_0\big\}.$$
Then $Q<I$ and $(I-Q)x\neq 0$. This implies that for any $ D\in\mathfrak D$, if $ D(I-Q)x\otimes y\in \mathfrak A_0$, then $D(I-Q)x=0$. Otherwise, we have that $(I-Q)D^*D(I-Q)x\otimes y$ is a nonzero element of $\mathfrak A_0$ for some $D\in \mathfrak D$. Hence, there exists a projection $P<I-Q$ in $\mathfrak D$ such that $Px\neq 0$ and $Px\otimes y\in\mathfrak A_0$. This yields a contradiction.

Set $x_0=(I-Q)x$. Then, if $Dx_0\otimes y\in\mathfrak A_0$ for some $D\in\mathfrak D$, it follows that $Dx_0=0$.  Similarly, we obtain an element $y_0$ such that if $x_0\otimes Dy_0\in\mathfrak{A}_0$ for some $D\in\mathfrak{D}$, then $Dy_0=0$, by using the fact that $(x_0\otimes y)U=x_0\otimes U^*y\in \mathfrak{A}_0$.
Note that
$$x_0=\bigoplus_{\lambda \in\Lambda}E_{\lambda}x_0 \quad \text{and} \quad y_0=\bigoplus_{\lambda \in\Lambda}E_{\lambda}y_0.$$
We may choose $\alpha<\beta$ in $\Lambda$ such that $E_{\alpha}x_0\neq0$, $E_{\beta}y_0\neq 0$, and $E_{\alpha}x_0\otimes E_{\beta}y_0\neq 0$ in $\operatorname{Alg}\operatorname{Lat}\mathfrak A$. It is clear that
\begin{equation}\label{f2}
\big\{ AE_{\alpha}x_0\otimes BE_{\beta}y_0:A,B \in\mathfrak D\big\}\cap \mathfrak A_0=\{0\}.
\end{equation}
Since $\mathfrak DE_{\lambda}=E_{\lambda}\mathcal B(\mathcal H)E_{\lambda}$ for any $\lambda \in\Lambda$, it follows that
\begin{equation}\label{f3}
[\mathfrak D E_{\alpha}x_0]=E_{\alpha}\mathcal H \quad \text{and} \quad [\mathfrak D E_{\beta}y_0]=E_{\beta}\mathcal H.
\end{equation}
On the other hand, we have $E_{\alpha}\mathfrak A E_{\beta}\subseteq \mathfrak A$. Then
\begin{equation}\label{f4}
E_{\alpha}\mathfrak A E_{\beta}=\{0\}
\end{equation}
by formulae (\ref{f2}) and (\ref{f3}).

Take any nonzero operator $A\in E_{\alpha}\mathcal B(\mathcal H)E_{\beta}=E_{\alpha}(\operatorname{Alg}\operatorname{Lat}\mathfrak A)E_{\beta}$.  Note that $A^2=0$. Then we have $I+A $ and $(I+A)^{-1}= I-A $ are elements of $ \operatorname{Alg}\operatorname{Lat}\mathfrak A$. Thus there exists an isometry $U\in \mathfrak D$ such that $U^*(I+A)$ and $(I-A)U$ are elements of $\mathfrak A$ by Lemma \ref{l1}. We similarly have  $U^*A, AU\in\mathfrak A$, and consequently $U^*E_{\alpha}A, AE_{\beta}U\in \mathfrak A$. If $E_{\alpha}$ (resp. $E_{\beta}$) is of finite rank, then $U^*E_{\alpha}$ (resp. $E_{\beta}U$) is a unitary operator on $E_{\alpha}\mathcal H$(resp.\ $E_{\beta}\mathcal H$). Thus
$$U^*E_{\alpha}A=E_{\alpha}U^*AE_{\beta} \quad (\text{resp. } AE_{\beta}U=E_{\alpha}AUE_{\beta})$$
is a nonzero operator in $E_{\alpha}\mathfrak A E_{\beta}$. This contradicts formula (\ref{f4}). If both $E_{\alpha}$ and $E_{\beta}$ are infinite dimensional, then we may take an injective operator $A\in E_{\alpha}\mathcal B(\mathcal H)E_{\beta}$ with dense range. We again have that $U^*E_{\alpha}A=E_{\alpha}U^*AE_{\beta}$ and $AE_{\beta}U=E_{\alpha}AUE_{\beta}$ are nonzero operators in $E_{\alpha}\mathfrak A E_{\beta}$, which is also a contradiction. Hence, $\mathfrak A=\operatorname{Alg}\operatorname{Lat}\mathfrak A$ and $\mathfrak A$ is reflexive.
\end{proof}
 Theorems \ref{t1}  and \ref{t3}  provide a complete answer to \cite[Question 6.3]{bha}.

It is well-known that a nest algebra $\operatorname{Alg}\mathcal N$ has the  factorization property if and only if $\mathcal N$ is  at most countable   by \cite[Theorem 4.7]{lar}.
 We next give a  description for a unital weak operator topology closed subalgebra $\mathfrak A$ to have   the factorization property.
\begin{corollary}
Let $\mathfrak A$ be a unital  closed subalgebra of $\mathcal B(\mathcal H)$  in weak operator topology. Then $\mathfrak A$ has the   factorization property if and only if  $\mathfrak A=\operatorname{Alg}\mathcal N$ for an at most countable nest.
\end{corollary}

\vskip10pt

{\bf AI Statement}.  The development of the mathematical arguments, particularly the formulation, refinement, and proof of several auxiliary lemmas, benefited from discussions with ChatGPT (GPT-6). All AI-assisted contributions have been carefully reviewed and verified by the authors, who assume full responsibility for the content of the manuscript.

\end{document}